\documentclass[MSNbibl,number,citesort,seceqn,dvips]{arxbj}
\usepackage{upgreek}
\usepackage{mathbh}
\aid{0}
\volume{19}
\issue{1}
\pubyear{2013}
\firstpage{115}
\lastpage{136}
\doi{10.3150/11-BEJ395}

\makeatletter

\newcommand{\eqref}[1]{(\ref{#1})}

\newcommand{\eps}{\varepsilon}
\newcommand{\E}{\mathbb{E}}
\newcommand{\R}{\mathbb{R}}
\newcommand{\ind}{\mathbh{1}}

\renewcommand{\P}{\mathbb{P}}
\newcommand{\dd}{\mathrm{d}}
\newcommand{\var}{\operatorname{var}}

\newtheorem{theorem}{Theorem}[section]
\newremark{rem}{Remark}[section]
\newtheorem{corollary}{Corollary}[section]
\newtheorem{proposition}{Proposition}[section]
\newtheorem{lemma}{Lemma}[section]

\makeatother

\begin{document}
\begin{frontmatter}

\title{Small time Chung-type LIL for L\'evy processes}
\runtitle{Chung-type LIL for L\'evy processes at zero}

\begin{aug}
\author[1]{\fnms{Frank} \snm{Aurzada}\corref{}\thanksref{1}\ead[label=e1]{aurzada@math.tu-berlin.de}},
\author[2]{\fnms{Leif} \snm{D\"oring}\thanksref{2}\ead[label=e2]{leif.doering@googlemail.com}} \and
\author[3]{\fnms{Mladen} \snm{Savov}\thanksref{3}\ead[label=e3]{savov@stats.ox.ac.uk}}
\runauthor{F. Aurzada, L. D\"oring and M. Savov} 
\address[1]{Technische Universit\"at Berlin, Institut f\"ur Mathematik,
Sekr. MA 7-4, Stra\ss e des 17. Juni 136, 10623 Berlin, Germany.
\printead{e1}}
\address[2]{Department of Statistics, University of Oxford, 1, South
Parks Road, Oxford OX1 3TG, UK.\hfill\break\printead{e2}}
\address[3]{New College, University of Oxford, Holywell Street, Oxford
OX1 3BN, UK.\\ \printead{e3}}
\end{aug}

\received{\smonth{1} \syear{2011}}
\revised{\smonth{8} \syear{2011}}

%
\begin{abstract}
We prove Chung-type laws of the iterated logarithm for general L\'evy
processes at zero. In particular, we provide tools to translate small
deviation estimates directly into laws of the iterated logarithm.

This reveals laws of the iterated logarithm for L\'evy processes at
small times in many concrete examples. In some cases, exotic norming
functions are derived.
\end{abstract}

%

\end{frontmatter}
%

\section{Introduction}\label{sec1}
A classical question in stochastic process theory is to understand
the asymptotic behavior of a given stochastic process $X=(X_t)_{t\geq
0}$ on the level of paths. In the present work, we consider general L\'
evy processes and find Chung-type LIL (laws of the iterated logarithm)
at zero; that is, given the L\'evy process $X$, we aim at
characterizing a norming function $b$, satisfying
\begin{equation}\label{LIL}
\liminf_{t\rightarrow0}\frac{\Vert X\Vert_t}{b(t)}=1,
\qquad\mbox{where }\Vert X\Vert_t:=\sup_{0\leq s\leq t}|X_s|.
\end{equation}
The topic of large and small time fluctuations of L\'evy processes has
been studied extensively in the past (see, e.g., Doney \cite{D}
for an overview and Bertoin \cite{B96}, Sato \cite{sato},
Bertoin, Doney and Maller \cite{BDM08}).

It is well known that, via the Borel--Cantelli lemma, Chung-type LIL
for a general stochastic process are connected to the so-called small
deviation rate of the process, that is,
\begin{equation}\label{small}
-\log\P(\Vert X\Vert_t\leq\eps),\qquad\mbox{as $\eps\to0$ and
$t\to0$}.
\end{equation}
The main motivation for this paper originates from the recent work
Aurzada and Dereich \cite{AD09}, where a framework for obtaining the
small deviation rate (\ref{small}) for general L\'evy processes (but
fixed $t$) is provided. The difficulty in passing over from the small
deviation estimate to the respective LIL concerns circumventing the
independence assumption of the Borel--Cantelli lemma.

In this paper we show how the asymptotics of (\ref{small}) imply
explicit LIL. We stress that it is not sufficient to have estimates for
(\ref{small}) for \textit{fixed} $t$, which usually are referred to as
small deviation estimates.

Small deviation problems are studied independently of LIL and have
connections to other fields, such as the approximation of stochastic
processes, coding problems, the path regularity of the process, limit
laws in statistics and entropy numbers of linear operators. We refer to
the surveys Li and Shao \cite{lishao}, Lifshits \cite{lif}, for an
overview of the field, and to
Lifshits \cite{sdbib}, for a regularly updated list of references,
which also includes references to laws of the iterated logarithm of
Chung type. The papers of Taylor \cite{taylor}, Mogul'ski\u\i\ \cite
{mogulskii}, Borovkov and Mogul'ski\u\i\ \cite{bormog},
Simon \cite{simon01,simonpvar}, Linde and Shi \cite{lindeshi},
Lifshits and Simon \cite{ls}, Linde and Zipfel \cite
{lindezipfel}, Shmileva \cite{elena}, Shmileva \cite{elena2} provide
a good source for
earlier results on small deviations of L\'evy processes.

We now discuss LIL for special L\'evy processes that have already
appeared in the literature. The norming function $b(t)=\sqrt{\uppi^2 t/
(8\log|\log t|)}$ for a standard Brownian motion can be derived from
the large time LIL, proved by Chung \cite{chung}, via time inversion.
For any L\'evy process with non-trivial Brownian component, the recent
result of Buchmann and Maller \cite{BM09} shows that (\ref{LIL}) holds
with the same norming function as for a standard Brownian motion. If
$X$ is an $\alpha$-stable L\'evy process, (\ref{LIL}) holds with
norming function $b(t)= ( c_\alpha t/\log|\log t|)^{1/\alpha}$, which
goes back to Taylor \cite{taylor}. The question was studied for
subordinators already in \cite{Fristedt}; there, the norming function can be
obtained from the Laplace transform.

Of course, it is natural to ask for the general structure of the
norming function for arbitrary L\'evy processes not having the special
features of the examples mentioned so far.
LIL for more general L\'evy processes were obtained in Wee \cite{W88};
see Wee \cite{wee2} for more examples. It was shown that if, for some
positive constant $\theta$,
\begin{equation} \label{eqn:weecondition}
\P(X_t>0)\geq\theta\quad\mbox{and}\quad\P(X_t<0)\geq\theta
\qquad\mbox{for all $t$
sufficiently small,}
\end{equation}
holds, then upper and lower bounds in the LIL hold in the following
sense: for $\lambda_1$ sufficiently small and $\lambda_2$
sufficiently large,
\[
1\leq\liminf_{t\rightarrow0}\frac{\Vert X\Vert_t}{b_{\lambda_1}(t)}
\quad\mbox{and}\quad\liminf_{t\rightarrow0}\frac{\Vert X\Vert
_t}{b_{\lambda_2}(t)}\leq1
\]
for norming functions $b_{\lambda}$ given by
\[
b_\lambda(t):=f^{-1} \biggl(\frac{\log|\log t|}{\lambda t} \biggr),
\]
where $f$ is given by some explicit, but complicated expression
depending on the L\'evy triplet.

Although the results of Wee are quite general, there are some points
which we aim to improve in the present work.
First, we try to demonstrate and explain clearly how the LIL follow
from small deviation estimates of type (\ref{small}) and which behavior
of the process is actually responsible for the correct norming function.
Second, we attempt to control the unspecified (and suboptimal)
constants $\lambda_1$ and $\lambda_2$ above, which can influence the
norming function essentially (see (\ref{eqn:constinexp}) below for an
example of influence on the exponential level) in the case when
$b_{\lambda}$ is not regularly varying at zero. In our approach, we
keep track of the appearing constants in an optimal way. This allows
us, in the case of known strong small deviation order, to transfer the
constant in the strong small deviation order to the limiting constant in
the LIL.
Third, we provide alternative conditions to~(\ref{eqn:weecondition})
which are explicit in terms of the L\'evy triplet. We believe our
conditions to be weaker than (\ref{eqn:weecondition}), but, as
necessary and sufficient conditions for the latter in terms of the L\'
evy triplet seem to be unknown in general, it is difficult to verify
our claim, although our examples hint at this direction.

This paper is structured as follows. In Section~\ref{sec:results}, we
give the main results that manage the transfer between small deviations
and LIL. Several examples of LIL for concrete L\'evy processes are
collected in Section~\ref{sec:examples}. The proofs are given in
Section~\ref{sec:proofs}.

Let us finally fix some notation. In this paper we let $X$ be a L\'evy
process with characteristic triplet $(\gamma, \sigma^2, \Pi)$, where
$\gamma\in\R$, $\sigma^2\geq0$, and the L\'evy measure $\Pi$ has no
atom at zero and satisfies
\[
\int(1\wedge x^2)\Pi(\dd x)<\infty.
\]
For basic definitions and properties of L\'evy processes we refer to
Bertoin \cite{B96}, Sato \cite{sato}. As we are interested only in
the behavior for small
times, we discard all jumps bigger than $1$ in absolute value and
assume such truncation \textit{throughout} the paper. Hence, the
characteristic exponent, $\E\mathrm{e}^{\mathrm{i} z X_t } =:
\mathrm{e}^{ t \psi(z)}$, has the form
\[
\psi(z)=\mathrm{i}\gamma z -\frac{\sigma^2z^2}{2}+\int
_{-1}^{1}(\mathrm{e}^{\mathrm{i}z x}-1-\mathrm{i}z
x)\Pi(\dd x),\qquad z\in\R.
\]
For later use we denote by $\Phi$ the Laplace exponent of a
subordinator $A$, $\E\mathrm{e}^{-u A_1}=\mathrm{e}^{-\Phi(u)}$,
\[
\Phi(u)= u \gamma_A + \int_0^\infty(1-\mathrm{e}^{-ux}) \Pi_A(\dd x).
\]

Further, we use the standard notation $\bar\Pi(\eps):=\Pi([-\eps
,\eps
]^c)$ for the two-sided tail of the L\'evy measure.

In the following, we denote by $f \sim g$ the strong asymptotic
equivalence, that is, $\lim f/g=1$, and by $f\approx g$ the weak
asymptotic equivalence, that is, $0<\liminf f/g \leq\limsup f/g <
\infty$.

\section{Main results} \label{sec:results}

Our first theorem manages the transfer from small deviation rates to
LIL under minimal loss of constants.
%
%
\begin{theorem}\label{t3}
Let $X$ be a L\'evy process (without loss of generality assume that
$X$ has jumps smaller than $1$ in absolute value). Let $F$ be a
function increasing to infinity at zero, such that with some $0<\lambda
_1\leq\lambda_2<\infty$
\begin{equation}\label{eqn:sdestmain}
\lambda_1 F(\eps)t\leq-\log\P(\Vert X\Vert_{t}<\eps) \leq
\lambda_2 F(\eps)t
\qquad\mbox{for all $\eps<\eps_0$ and $t<t_0$.}
\end{equation}
Further, define
\[
b_\lambda(t):=F^{-1} \biggl(\frac{\log|\log t|}{\lambda t} \biggr)
\]
for $\lambda>0$, and assume that, as $n\to\infty$,
\begin{eqnarray} \label{eqn:conditionM}
&&(n+1)^{-(n+1)^\beta} \biggl| \int_{|x|>b_{\lambda_2'}(n^{-n^\beta
})} x
\Pi(\dd x) - \gamma\biggr|\nonumber\\ [-8pt]\\ [-8pt]
&&\quad=\mathrm{o} (b_{\lambda_2'} (n^{-n^\beta} ) )\qquad\mbox
{for all
$\beta>1$ and $\lambda_2'>\lambda_2$.}\nonumber
\end{eqnarray}
Then the LIL
\[
1\leq\liminf_{t\rightarrow0}\frac{\Vert X\Vert_{t}}{b_{\lambda_1'}(t)}
\quad\mbox{and}\quad\liminf_{t\rightarrow0}\frac{\Vert X\Vert
_{t}}{b_{\lambda
_2'}(t)}\leq1
\]
hold almost surely for any $\lambda_1'<\lambda_1$ and $\lambda
_2'>\lambda_2$.
\end{theorem}

%
\begin{rem}\label{RemMladen1}
It is important to note the role of \eqref{eqn:conditionM}. It
ensures that the process does not become too asymmetric when one
continues to cut off more and more smaller jumps. Only in this case is
it possible to expect an estimate of type \eqref{eqn:sdestmain} to
follow from the framework given in Aurzada and Dereich \cite{AD09}.
Corollary~\ref{cort2} below and, in particular,
\eqref{eqn:cond-esschervanishes} give a sufficient condition when this
is the case.
\end{rem}

%
%

%
\begin{rem}\label{RemMladen3}
Let us relate our condition \eqref{eqn:conditionM} with the
condition of Wee \cite{W88}. Note that \eqref{eqn:conditionM} is
analytic, that is, in terms of the L\'evy triplet, whereas Wee's
condition \eqref{eqn:weecondition} is probabilistic. It seems that
\eqref{eqn:weecondition} cannot always be checked from the L\'evy
triplet. To understand the difficulty, it may be instructive to look at
Theorems~4 and~5 in Andrew \cite{A08}, which reformulate \eqref
{eqn:weecondition} in terms of other probabilistic quantities. 
\end{rem}


It is crucial that there is almost no loss of constants in the
transfer from the small deviations to the LIL as in cases when
$b_{\lambda}$ is not regularly varying, the constants $\lambda_1',
\lambda_2'$ may influence the rate function drastically; see (\ref
{eqn:constinexp}) for an extreme example.

If instead $b_{\lambda}$ only depends on $\lambda$ via a
multiplicative constant, our approach allows to strengthen the previous
theorem to the optimal limiting constants. Such examples occur, for
instance, if the small deviation rate function $F$ is regularly varying.

\begin{corollary}\label{cor:t1corollary}
In the setting of Theorem~\ref{t3}, assume additionally that $F$ is
regularly varying at zero with non-positive exponent. Then the
following LIL hold almost surely:
\begin{equation} \label{eqn:lilnonre}
1\leq\liminf_{t\rightarrow0}\frac{\Vert X\Vert_{t}}{b_{\lambda_1}(t)}
\quad\mbox{and}\quad\liminf_{t\rightarrow0}\frac{\Vert X\Vert
_{t}}{b_{\lambda
_2}(t)}\leq1.
\end{equation}
In particular, if there is $\lambda>0$ such that (\ref{eqn:sdestmain})
holds for all $\lambda_1<\lambda$ and all $\lambda_2>\lambda$, then
\[
\liminf_{t\rightarrow0}\frac{\Vert X\Vert_{t}}{b_{\lambda
}(t)}=1\qquad\mbox{a.s.}
\]
\end{corollary}

In the setting of a regularly varying rate function, say $F$ is
regularly varying at zero with exponent $-\alpha$, $\alpha>0$, one can
express (\ref{eqn:lilnonre}) as
\[
\liminf_{t\rightarrow0}\frac{\Vert X\Vert_{t}}{b_{1}(t)}\in
[\lambda
_1^{1/\alpha},\lambda_2^{1/\alpha} ],\qquad\mbox{a.s.}
\]
This shows that only the quality of the small deviation estimate (\ref
{eqn:sdestmain}) matters in order to obtain the limiting constant in
the LIL. Recall that the Blumenthal zero--one law implies that the limit
is almost surely equal to a deterministic constant, which in this case
can be specified.
%

Theorem~\ref{t3} reduces the question of the right norming function
for the LIL to the question of small deviations which is known
precisely for many examples.
For general L\'evy processes, those have been obtained in Aurzada and Dereich
\cite{AD09} (their results were stated for $t=1$ only, but hold, in
general, as we discuss in Proposition~\ref{prop:ad} below). In
particular, for symmetric L\'evy processes, their main result states
that the rate function is given by
\begin{equation}\label{a2}
F(\eps)=\eps^{-2}U(\eps),
\end{equation}
where $U(\eps)$ is the variance of $X$ with jumps larger than $\eps$
replaced by jumps of size $\eps$,
\begin{equation}\label{eqn:defnU}
U(\eps):=\eps^2\bar\Pi(\eps)+\sigma^2+\int_{-\eps}^{\eps
}x^2\Pi(\dd x).
\end{equation}
From these specific small deviations we can deduce the following
corollary for symmetric processes.

\begin{corollary} \label{cor:sddirectsymmetric}
Let $X$ be a symmetric L\'evy process; then there are $0<\lambda_1\leq
\lambda_2<\infty$ such that,
almost surely,
\[
1\leq\liminf_{t\rightarrow0}\frac{\Vert X\Vert_{t}}{b_{\lambda_1}(t)}
\quad\mbox{and}\quad\liminf_{t\rightarrow0}\frac{\Vert X\Vert
_{t}}{b_{\lambda
_2}(t)}\leq1,
\]
with
\[
b_\lambda(t):=F^{-1} \biggl(\frac{\log|\log t|}{\lambda t} \biggr)
\]
and $F$ defined in (\ref{a2}). If, additionally, $F$ is regularly
varying at zero with exponent $-\alpha$, $\alpha>0$, then the following
general bounds hold:
\[
\frac{1}{12}\frac{1}{2^{\alpha}}\leq\lambda_1\leq\lambda_2\leq
3^{\alpha}10.
\]
\end{corollary}

The loss of constants in the corollary is only due to the general
formulation. For some examples we will see below that the small
deviations are known in the strong asymptotic sense so that
Theorem~\ref
{t3} gives the precise law.

In the sequel we call ``strongly non-symmetric'' L\'evy processes the
processes for which (\ref{eqn:conditionM}) does not hold. Their study
requires different assumptions on $b_{\lambda}$; see \eqref
{eqn:regularityofb}. For this case, we provide a different link between
small deviation rates and LIL. The next result does not require \eqref
{eqn:conditionM} and thus allows us to study the ``strongly
non-symmetric'' L\'evy processes as well as other cases when \eqref
{eqn:conditionM} is difficult to verify. The latter is substituted by
the seemingly easier \eqref{eqn:regularityofb} at the expense of the
strength of the result; that is, we manage to keep track of the
constants in the norming function in an optimal way, but lose the
limiting constant. We have tried unsuccessfully to find a suitable
relation between \eqref{eqn:conditionM} and \eqref{eqn:regularityofb}.
We strongly suspect that neither one follows from the other.

%
%
\begin{theorem}\label{t}
Let $X$ be a L\'evy process with jumps smaller than $1$ in absolute
value, and let $F$ be a function increasing to infinity at zero such
that for $0<\lambda_1\leq\lambda_2<\infty$
\begin{equation} \label{eqn:yetanothersdestimate}
\lambda_1 F(\eps)t\leq-\log\P(\Vert X\Vert_{t}<\eps) \leq
\lambda_2 F(\eps)t\qquad
\mbox{for all $\eps<\eps_0$ and $t<t_0$.}
\end{equation}
Furthermore, set
\begin{equation} \label{eqn:quantitythebee}
b_\lambda(t):=F^{-1} \biggl(\frac{\log|\log t|}{\lambda t} \biggr),
\end{equation}
and suppose that there is a constant $C>0$ such that
\begin{equation} \label{eqn:regularityofb}
C b_{\lambda}(t)\leq b_{\lambda}(t/2),\qquad0<t\leq t_0, \lambda\in
(\lambda_1/2,2\lambda_2).
\end{equation}
Then the LIL
\[
0<\liminf_{t\rightarrow0}\frac{\Vert X\Vert_{t}}{b_{\lambda_1'}(t)}
\quad\mbox{and}\quad\liminf_{t\rightarrow0}\frac{\Vert X\Vert
_{t}}{b_{\lambda
_2'}(t)}<\infty
\]
hold almost surely for all $\lambda_1'<\lambda_1$ and $\lambda
_2<\lambda_2'$.
\end{theorem}

Again, if the rate function $F$ is regularly varying, then we can
strengthen the result. Recall that the L\'evy processes that appear in
the formulation of the next sequence of results have jumps smaller than
$1$ in absolute value.

\begin{corollary}\label{cor:t1corollary2}
In the setting of Theorem~\ref{t}, assume additionally that $F$ is
regularly varying at zero with negative exponent. Then the following
LIL holds almost surely:
\[
\liminf_{t\rightarrow0}\frac{\Vert X\Vert_{t}}{b_{1}(t)} \in
(0,\infty).
\]
\end{corollary}

The theorems listed so far manage the transfer between small deviation
order and LIL. Similarly to Corollary~\ref{cor:sddirectsymmetric}, we
can combine them with the main results of Aurzada and Dereich \cite
{AD09}. This looks more technical in the present case. We give an
explanation of the role of the different terms after stating the result.

\begin{corollary}\label{cort2}
Let $X$ be a L\'evy process with triplet $(\gamma, \sigma^2, \Pi)$.
Assume that $u_\eps$ is the solution of the equation $\Lambda_\eps
'(u)=0$, where $\Lambda_\eps$ is the following log Laplace transform:
\begin{equation} \label{eqn:quantitylambda}
\Lambda_\eps(u)=\frac{\sigma^2}{2} u^2 + \biggl(\gamma-\int
_{[-1,1]\setminus[-\eps,\eps]} x \Pi(\dd x) \biggr) u + \int
_{-\eps}^{\eps}
(\mathrm{e}^{u x}-1 - u x) \Pi(\dd x).
\end{equation}
Set
\begin{equation} \label{eqn:quantitythef}
F(\eps):=\eps^{-2} U_\eps(\eps)-\Lambda_\eps(u_\eps),\qquad
U_\eps(\eps
):=\eps^2\bar\Pi(\eps)+\sigma^2+\int_{-\eps}^{\eps}x^2 \mathrm
{e}^{-u_\eps x}
\Pi(\dd x),
\end{equation}
and assume $F$ is increasing to infinity as $\eps\to0$. Define $b$ as
in (\ref{eqn:quantitythebee}), and assume that $b$ satisfies (\ref
{eqn:regularityofb}).
If, furthermore,
\begin{equation} \label{eqn:cond-esschervanishes}
\eps|u_\eps| = \mathrm{o}( \log\log F(\eps) ),\qquad\mbox{as
$\eps\to0$,}
\end{equation}
is satisfied, then we have, for some $\lambda_1, \lambda_2>0$,
\[
0< \liminf_{t\rightarrow0}\frac{\Vert X\Vert_{t}}{b_{\lambda_1}(t)}
\quad\mbox{and}\quad\liminf_{t\rightarrow0}\frac{\Vert X\Vert
_{t}}{b_{\lambda
_2}(t)}<\infty\qquad\mbox{a.s.}
\]
\end{corollary}

Let us explain the quantities appearing in Corollary~\ref{cort2} in
more detail. The main observation is that the proof for the small
deviation estimates in Aurzada and Dereich \cite{AD09} (Theorem~1.5)
can be used directly for any $t>0$ to obtain the following proposition.

\begin{proposition} \label{prop:ad}
Let $\Lambda_\eps$ be as defined in (\ref{eqn:quantitylambda}) and
assume that $u_\eps$ is the solution of $\Lambda_\eps'(u_\eps)=0$.
Then, with $F$ as in (\ref{eqn:quantitythef}), we have, for all $t>0$
and all $\eps<1$,
\begin{equation} \label{eqn:sdquantityadstrengthend}
\frac{1}{12} t F(2\eps) - \eps|u_{2\eps}| -1 \leq- \log\P( \Vert
X\Vert_t
\leq\eps) \leq10 t F \biggl(\frac{\eps}{3} \biggr) + \eps
|u_{\eps/3}|+3.
\end{equation}
\end{proposition}

The term $\bar\Pi(2\eps)$ in (\ref{eqn:sdquantityadstrengthend})
(included in the $F$ term) comes from the requirement that there should
be no jumps larger than $2\eps$. After removing these jumps, the
process may drift out of the interval $[-\eps,\eps]$, which is
prevented by applying an Esscher transform to the process, whose
``price'' is given by the term $-\Lambda_\eps(u_\eps)$. The quantity
$u_\eps$ is the drift that has to be subtracted in order to make the
process a martingale. Then the remaining process is treated as in the
symmetric case, and the same term $\eps^{-2} U_\eps(\eps)$ appears as
in (\ref{a2}), but, this time, with respect to the L\'evy measure
transformed by the change of measure.


Note that (\ref{eqn:sdquantityadstrengthend}) is almost the required
estimate in (\ref{eqn:yetanothersdestimate}), except for the term
$\eps
|u_\eps|$, which may spoil the estimate. It is exactly condition (\ref
{eqn:cond-esschervanishes}) that ensures that the term $\eps|u_\eps|$
can be neglected.

We stress that in some cases $\eps|u_\eps|$ does give an order that
is larger than $t F(\eps)$ so that the function $b$ from (\ref
{eqn:quantitythebee}) is not the right norming function. This effect
can be observed in some examples below. In particular, this happens for
processes of bounded variation with non-zero drift.

\begin{proposition} \label{prop:bvnodrift}
Let $X$ be a L\'evy process with bounded variation and non-vanishing
effective drift, that is, $\int_{[-1,1]} |x| \Pi(\dd x)< \infty$ and
$c:=\gamma-\int_{-1}^1 x \Pi(\dd x)\neq0$. Then
\[
\lim_{t\rightarrow0}\frac{\Vert X\Vert_t}{t} =|c|\qquad\mbox{a.s.}
\]
\end{proposition}

The proof of this proposition is based on classical arguments rather
than any connection to small deviations.


\section{Explicit LIL for L\'{e}vy processes}\label{sec:examples}
In this section, we collect concrete L\'evy processes for which we can
transform small deviation results to an LIL. As we have seen,
understanding the small deviation rates is crucial. \medskip

In this section we keep in mind that our processes in all proofs have
no jumps bigger than $1$ in absolute value. However, without loss of
generality, in some statements we use ``stable L\'evy processes'' and
others which presuppose unbounded jumps.

The first corollary gives us a useful variance domination principle
for LIL that works for many examples.

\begin{corollary}\label{cor:domination}
Suppose $X^1$ and $X^2$ are independent symmetric L\'evy processes,
then $X^1+X^2$ and $X^2$ fulfill precisely the same LIL if
\[
\lim_{\eps\to0}\frac{U_{X^1}(\eps)}{U_{X^2}(\eps)}=0.
\]
\end{corollary}

\begin{pf}
This follows directly from Corollary~\ref{cor:sddirectsymmetric}
noticing that $U_{X^1+X^2}=U_{X^1}+U_{X^2}$.
\end{pf}

In the same spirit, the following corollary (recovering (3.2) in
Buchmann and Maller \cite{BM09}) displays the intuitive fact that a
non-zero Brownian
component dominates the jumps of a L\'evy process.

\begin{corollary}\label{cor:withbrownian}
If $X$ is a L\'evy process with $\sigma\neq0$, then
\[
\liminf_{t\rightarrow0}\frac{\Vert X\Vert_t}{\sqrt{t/\log| \log
t|}}=\frac
{\uppi\sigma}{\sqrt{8}}\qquad\mbox{a.s.}
\]
\end{corollary}

\begin{pf}
Following precisely the proof of Corollary~2.6 of Aurzada and Dereich
\cite{AD09}, one can show that the small deviation rates of L\'evy
processes with non-zero Brownian component are given by
\[
-\log\P(\Vert X\Vert_{t}<\eps)\sim\frac{\uppi^2\sigma^2}{8}\eps
^{-2}t,\qquad\mbox{as
$\eps\to0$ and $t\to0$.}
\]
Hence, the norming function follows from Theorem~\ref{t3}. As the
process is not necessarily symmetric, condition (\ref{eqn:conditionM})
has to be checked: Since $b(t)=\sqrt{t\uppi^2/ (8\log|\log t|)}$ and
$\int_{|x|>\eps} |x| \Pi(\dd x)=\mathrm{o}(\eps^{-1})$, it remains
to be seen that
\[
a_{n+1} \leq c b(a_n)^2= a_n / \log|\log a_n|
\]
for $a_n=n^{-n^\beta}$ and $\beta>1$. This can be verified by simple
computations.
\end{pf}

Similarly to L\'evy processes with non-zero Brownian component,
symmetric processes of smaller small deviation order (e.g.,\ stable
processes of smaller index) are dominated by stable L\'evy processes.

\begin{corollary}\label{cor:stables}
Let $X$ be a symmetric $\alpha$-stable L\'evy process with $\alpha\in
(0,2]$, and let $Y$ be symmetric with $U_Y(x)=\mathrm{o}(x^{2-\alpha
})$. Then
there is a constant $0<c_\alpha<\infty$ such that
\[
\liminf_{t\rightarrow0}\frac{\Vert X+Y\Vert_t}{(t/\log|\log
t|)^{1/\alpha
}}=\liminf_{t\rightarrow0}\frac{\Vert X\Vert_t}{(t/\log|\log
t|)^{1/\alpha
}}=c_\alpha^{1/\alpha}\qquad\mbox{a.s.}
\]
\end{corollary}

\begin{pf}
The small deviation rate is given by
\[
-\log\P(\Vert X\Vert_{t}<\eps)\sim c_{\alpha} \eps^{-\alpha
}t,\qquad\mbox{as
$\eps\to0$ and $t\to0$}
\]
for some constant $c_{\alpha}>0$ (see, e.g., page 220 in Bertoin \cite{B96}).
Hence, the LIL follows from Corollary~\ref{cor:t1corollary} and
Corollary~\ref{cor:domination}.
\end{pf}

%
%
\begin{rem}
The constant $c_\alpha$ in the LIL of stable L\'evy processes is the
unknown constant of the small deviations for respective $\alpha$-stable
L\'evy processes (see Taylor \cite{taylor} and Proposition~3 and Theorem~6
in Chapter VIII of Bertoin \cite{B96}). The results of Aurzada and Dereich
\cite{AD09} entail the following concrete bounds:
\[
\frac{2 C }{2^{\alpha}} \biggl( \frac{1}{\alpha} + \frac
{1}{12(2-\alpha)} \biggr)
< c_\alpha< 3^\alpha\cdot2 C \biggl( \frac{1}{\alpha} + \frac
{10}{2-\alpha}
\biggr),
\]
where $C$ is the constant in the L\'evy measure: $\Pi(\dd x)=C
|x|^{-(1+\alpha)}\,\dd x$. This implies $c_\alpha\sim2 C/\alpha$, as
$\alpha\to0$. We remark that, contrary to the symmetric case, the
constant $c_\alpha$ is known explicitly for completely asymmetric
stable L\'evy processes; see Bertoin \cite{bertoin96}.
\end{rem}

Let us study the case when $\Pi$ behaves as a regularly varying
function at zero and is symmetric. Then the following LIL are satisfied.

\begin{corollary}
Let $X$ be a L\'evy process with triplet $(0,0,\Pi)$ with $\Pi$ being
symmetric and
\[
\bar\Pi(\eps)\approx\eps^{-\alpha} |\log\eps|^{-\gamma},\qquad
\mbox{as
$\eps\to0$,}
\]
with $0<\alpha<2$ or $\alpha=2, \gamma>1$. Then
\[
\liminf_{t\rightarrow0}\frac{\Vert X\Vert_{t}}{b(t)}\in(0,\infty
)\qquad\mbox{a.s.}
\]
with
\[
b(t)=
\cases{\displaystyle
\biggl(\frac{t|\log t|^{-\gamma}}{\log|\log t|} \biggr)^{1/\alpha
},&\quad
$0<\alpha<2$, \vspace*{2pt}\cr
\displaystyle\biggl(\frac{t|\log t|^{1-\gamma}}{\log|\log t|}
\biggr)^{1/2}, &\quad$\alpha
=2, \gamma>1$.
}
\]
\end{corollary}

\begin{pf}
The corollary follows from Theorem~\ref{t3}. The required small
deviation estimate,
\[
-\log\P(\Vert X\Vert_{t}<\eps)\approx
\cases{
\eps^{-\alpha}|\log\eps|^{-\gamma} t,&\quad$0<\alpha<2$,\vspace
*{2pt}\cr
\eps^{-2}|\log\eps|^{1-\gamma}t,&\quad$\alpha=2, \gamma>1$,
}
\]
as $\eps\to0$ and $t\to0$, is obtained from Proposition~\ref{prop:ad}
(cf. Example~2.2 in Aurzada and Dereich \cite{AD09} for $t=1$). Since
we deal with a
symmetric process, condition (\ref{eqn:cond-esschervanishes}) is
trivially satisfied due to $u_{\eps}=0$.
\end{pf}

Having discussed the $\alpha$-stable like cases, we now consider L\'
evy processes with polynomial tails near zero of \textit{different}
exponents. The technique used for this example can be extended to any
case with essentially regularly varying L\'{e}vy measure at zero. Let
$X$ be a L\'evy process with triplet $(\gamma,0,\Pi)$, where $\Pi$ is
given by
%
%
\begin{equation}\label{eq:regularlm}
\frac{\Pi(\dd x)}{\dd x} = \frac{C_1 \ind
_{(0,1]}(x)}{x^{1+\alpha_1}} + \frac{C_2 \ind
_{[-1,0)}(x)}{(-x)^{1+\alpha_2}},
\end{equation}
with $2>\alpha_1\geq\alpha_2$ and $C_1,C_2\geq0$, $C_1+C_2\neq0$. We
now analyze the pathwise behavior at zero in the cases when $\alpha
_1>1$, $\alpha_1=1$, and $0<\alpha_1<1$, respectively. The second
exponent $\alpha_2$ can be even negative.

\begin{corollary}\label{pol} Let $X$ be a L\'evy process with triplet
$(\gamma,0,\Pi)$ with $\Pi$ as in (\ref{eq:regularlm}). Then the
following holds:
\begin{enumerate}
\item If $\alpha_1\geq\alpha_2$, $C_1\neq0$, and $\alpha_1>1$, then
\[
\liminf_{t\rightarrow0}\frac{\Vert X\Vert_t}{(t / \log|\log
t|)^{1/\alpha
_1}}\in(0,\infty)\qquad\mbox{a.s.}
\]
%
%
\item If $\alpha_1=\alpha_2=1$ and $C_1=C_2$, then
\[
\liminf_{t\rightarrow0}\frac{\Vert X\Vert_t}{t / \log|\log t|}\in
(0,\infty
)\qquad\mbox{a.s.}
\]
%
%
\item If $1>\alpha_1\geq\alpha_2$ and the effective drift does not
vanish, then
\[
\lim_{t\rightarrow0}\frac{\Vert X\Vert_t}{t} = |c|\qquad\mbox{a.s.}
\]
\end{enumerate}
\end{corollary}

\begin{pf}
Parts 1 and 2 follow from Theorem~\ref{t3}. The required small
deviation estimates,
\[
-\log\P(\Vert X\Vert_{t}<\eps)\approx\eps^{-\alpha_1}t
\]
for $\eps\to0$ and $t\to0$, are obtained from Proposition~\ref
{prop:ad} (cf.\ Corollary~2.7,~2.8 and~2.9 of Aurzada and Dereich \cite
{AD09} for
$t=1$; note that $u_\eps\approx\eps^{-1}$ in all cases). One can
easily check condition (\ref{eqn:cond-esschervanishes}).

In part 3 the process is of bounded variation, so that the claim is
included in Proposition~\ref{prop:bvnodrift}.
\hspace*{1pt}
\end{pf}

We now come to L\'evy processes obtained from Brownian motion by
subordination, that is, $X_t=\sigma B_{A_t}$, where $B$ is a Brownian
motion independent of the subordinator $A$. In this case, the resulting
L\'evy process is symmetric and the small deviation asymptotics is
governed by the truncated variance $U$ from (\ref{eqn:defnU}).

\begin{corollary}
Let $B$ be a Brownian motion independent of the subordinator $A$, where
$A$ has Laplace exponent $\Phi$. For $\lambda>0$ we set $b_\lambda
(t):=F^{-1} (\frac{\log|\log t|}{\lambda t} )$ with
\[
F(\eps):= \Phi(\sigma^2\eps^{-2}) + \gamma_A \sigma^2 \eps^{-2}.
\]
Then, for some $\lambda_1, \lambda_2>0$,
\[
1\leq\liminf_{t\rightarrow0}\frac{\Vert X\Vert_t}{b_{\lambda_1}(t)}
\quad\mbox{and}\quad\liminf_{t\rightarrow0}\frac{\Vert X\Vert
_t}{b_{\lambda
_2}(t)}\leq1\qquad\mbox{a.s.}
\]
In particular, if $\gamma_A=0$ and $\Phi$ is regularly varying with
positive exponent, we have
\[
\liminf_{t\rightarrow0}\frac{\Vert X\Vert_t}{(\Phi^{-1} (\log|
\log t| / t
))^{-1/2}}\in(0,\infty)\qquad\mbox{a.s.}
\]
\end{corollary}

\begin{pf}
The corollary follows from Theorem~\ref{t3} with the small deviation
estimate from Proposition~\ref{prop:ad},
\[
-\log\P( \Vert X\Vert_t \leq\eps) \approx\bigl(\Phi(\sigma
^2\eps^{-2}) + \gamma
_A \sigma^2 \eps^{-2}\bigr)t,
\]
as $\eps\to0$ and $t\to0$ (cf. Example~2.13 of Aurzada and Dereich
\cite{AD09} for
$t=1$ and note the misprint there). Condition (\ref{eqn:conditionM}) is
trivially fulfilled as the process is symmetric.
\end{pf}

For a more specific example, in particular, exhibiting exotic small
time behavior, we choose the subordinator $A$ to be a Gamma process.
Then one defines the so called Variance-Gamma process as
\[
X_t=\sigma B_{A_t}+\mu A_t
\]
for some constants $\sigma\neq0$ and $\mu\in\R$.

\begin{corollary}
Let $X$ be a Variance-Gamma process; then for $\mu=0$ there are some
constants $0<\lambda_1\leq\lambda_2<\infty$ such that
\begin{equation} \label{eqn:constinexp}
1\leq\liminf_{t\rightarrow0}\frac{\Vert X\Vert_t}{ \mathrm{e}^{ -
\lambda_1 \log|\log
t|/t}}\quad\mbox{and}\quad\liminf_{t\rightarrow0}\frac{\Vert
X\Vert_t}{ \mathrm{e}^{- \lambda_2
\log|\log t|/t}}\leq1\qquad\mbox{a.s.},
\end{equation}
whereas for $\mu\neq0$
\[
\liminf_{t\rightarrow0}\frac{\Vert X\Vert_t}{t}=|\mu| \E(
A_1)\qquad\mbox{ a.s.}
\]
\end{corollary}

\begin{pf}
The second part is included in Proposition~\ref{prop:bvnodrift}, since
the process is of bounded variation with non-zero effective drift. In
the first part, the effective drift is zero, and the claim follows from
Theorem~\ref{t3}. The small deviation estimate,
\[
-\log\P( \Vert X\Vert_t \leq\eps) \approx t |\log\eps|,\qquad
\mbox{as $\eps\to0$
and $t\to0$},
\]
follows from Proposition~\ref{prop:ad} (cf. Example 2.12 of Aurzada
and Dereich \cite{AD09} for $t=1$).
\end{pf}

%

In the first case of the previous corollary, the dependence of good
small deviation estimates and good LIL becomes transparant. The fact
that we cannot specify the constants $\lambda_1, \lambda_2$ in (\ref
{eqn:constinexp}) is only caused by the weak asymptotics for the small
deviation estimate as we do not lose any further constants in the
transfer of small deviations to the LIL. If one does not have more
control on the constants $\lambda_1, \lambda_2$, the understanding of
the precise small time behavior of $X$ is far from optimal as the error
enters exponentially.

\section{Proofs} \label{sec:proofs}


We start with a lemma which shows that the small deviation order is at
least as large as the term induced by the variance, defined in (\ref
{eqn:defnU}).
\begin{lemma}
Let $\eps>0$, and let $X$ be a L\'evy process with L\'evy measure
concentrated on $[-\eps,\eps]$, then
\[
\P( \Vert X\Vert_t \leq\eps/2 )\leq\exp\biggl( - \eps^{-2}
\biggl(\int_{-\eps}^\eps x^2
\Pi(\dd x) + \sigma^2 \biggr) t\big/12+1 \biggr)\qquad\mbox{for
$t\geq0$.}
\]
\end{lemma}

\begin{pf}
We proceed similarly to Lemma~4.2 in~Aurzada and Dereich \cite{AD09}.
Let $\tau$ be the first exit time of $X$ out of $[-\eps,\eps]$. Then,
by Wald's identity,
\begin{eqnarray*}
4 \eps^2 & \geq&\limsup_{t\to\infty} \E[ X^2_{t\wedge\tau}]
\geq
\limsup_{t\to\infty} \var[ X_{t\wedge\tau}] \\
&=& \limsup_{t\to
\infty} \biggl(\int_{-\eps}^\eps x^2 \Pi(\dd x) + \sigma^2 \biggr
) \E[ t\wedge\tau]
= \biggl(\int_{-\eps}^\eps x^2 \Pi(\dd x) + \sigma^2 \biggr) \E[
\tau].
\end{eqnarray*}
Therefore,
\[
\P\biggl(\tau\geq8 \eps^2 \big/ \biggl(\int_{-\eps}^\eps x^2
\Pi(\dd x) + \sigma^2 \biggr) \biggr)
\leq\frac{ (\int_{-\eps}^\eps x^2 \Pi(\dd x) + \sigma^2 ) \E
[\tau]}{8
\eps^2} \leq\frac12.
\]
Let $n:=\lfloor t(\int_{-\eps}^\eps x^2 \Pi(\dd x) + \sigma^2)/(8
\eps
^2) \rfloor$, and set $t_i := 8i \eps^2 / (\int_{-\eps}^\eps x^2
\Pi(\dd
x) + \sigma^2)$, $i=0,\ldots, n$. Then
\[
\P( \Vert X\Vert_t \leq\eps/2) \leq\P\Bigl( \forall i=0, \ldots
, n-1 \dvt\sup
_{s\in[t_i,t_{i+1})} |X_s -X_{t_i}|\leq\eps\Bigr) = \P( \tau\geq t_1)^n
\leq2^{-n}.
\]
\upqed
\end{pf}

This shows that the small deviation order is always at least as large
as the term induced by the truncated variance process. This fact will
be needed later on.

\begin{lemma}\label{lem:flargeru}
Let $F$ be a function that increases to infinity at zero. If, for some
L\'evy process $X$, for $t\leq t_0$ and $\eps<\eps_0$,
\[
-\log\P( \Vert X\Vert_t \leq\eps) \leq F(\eps) t,
\]
then, for some absolute constant $c>0$ and all $\eps>0$ small enough,
\[
\eps^{-2} U(\eps) \leq c \bigl(F(\eps)+1\bigr).
\]
\end{lemma}

\begin{pf}
We use the assumption together with the fact that if $\Vert X\Vert
_t\leq\eps
$, then $X$ must not have jumps larger than $2\eps$ and the previous lemma,
\[
\mathrm{e}^{-F(\eps)t} \leq\P(\Vert X\Vert_t\leq\eps) = \mathrm
{e}^{-\bar\Pi(2\eps) t} \P
(\Vert X'\Vert_t\leq\eps) \leq\mathrm{e}^{-\bar\Pi(2\eps) t}
\mathrm{e}^{- (2\eps)^{-2}
(\int_{-2\eps}^{2\eps} x^2 \Pi(\dd x) + \sigma^2 ) t/12+1},
\]
where $X'$ has L\'evy measure $\Pi$ restricted to $[-2\eps,2\eps]$.
Noting that Lemma~5.1 of Aurzada and Dereich \cite{AD09} implies that
$U(\eps)/\eps^2 \approx U(2\eps)/(2\eps)^2$, the statement of the lemma
is proved.
\end{pf}

The lower bound in the LIL comes from the following lemma.

\begin{lemma}\label{lem:lower}
Let $F$ be a function that increases to infinity at zero such that,
for all $t\leq t_0$ and $\eps\leq\eps_0$,
\[
\lambda F(\eps)t \leq-\log\P(\Vert X\Vert_t\leq\eps),
\]
and, for $\lambda>0$, we set $b_\lambda(t):=F^{-1} (\frac{\log|\log
t|}{\lambda t} )$. Then, for any $\lambda'<\lambda$,
\[
1\leq\liminf_{t\rightarrow0}\frac{\Vert X\Vert_t}{b_{\lambda
'}(t)}\qquad\mbox{a.s.}
\]
\end{lemma}

\begin{pf}
For any $\lambda'<\lambda$, we can find $0<r<1$ such that $1< \lambda
r/\lambda'$. Note that
\[
\sum_{n}\P\bigl(\Vert X\Vert_{r^{n+1}}\leq b_{\lambda'}(r^n) \bigr
)<\infty
\]
since
\begin{equation}\label{est}
-\log\P\bigl(\Vert X\Vert_{r^{n+1}}\leq b_{\lambda'}(r^n) \bigr
)\geq\lambda F(
b_{\lambda'}(r^n)) r^{n} r = \lambda\frac{r}{\lambda'} \log|\log
r^n| =
\log n^{r \lambda/\lambda'} + {\rm const.}
\end{equation}
Hence, by the Borel--Cantelli lemma,
\[
\{n\dvt\Vert X\Vert_{r^{n+1}}\leq b_{\lambda'}(r^{n}) \}
\]
is almost surely a finite set. Thus, for each path $\omega$, we have
that, for any $n\geq n_{0}(\omega)$ and any $t\in[r^{n+1},r^{n})$,
\[
\frac{\Vert X\Vert_{t}}{b_{\lambda'}(t)}\geq\frac{\Vert X\Vert
_{r^{n+1}}}{b_{\lambda
'}(r^{n})}\geq1,
\]
as $b_{\lambda'}$ is an increasing function. We take $\liminf_{t\to0}$
to obtain the statement.
\end{pf}

The proof of the upper bound in the LIL requires the following lemma.
\begin{lemma}\label{L2}
Let $F$ be a function that increases to infinity at zero such that for
all $t\leq t_0$ and $\eps\leq\eps_0$
\[
-\log\P(\Vert X\Vert_t\leq\eps) \leq\lambda F(\eps) t
\]
and, for $\lambda>0$, set $b_\lambda(t):=F^{-1} (\frac{\log|\log
t|}{\lambda t} )$. Assume that
\begin{equation}\label{mladen}
\limsup_{n\to\infty}\frac{\Vert X\Vert_{(n+1)^{-(n+1)^{\beta
}}}}{b_{\lambda}
(n^{-n^{\beta}} )}=0\qquad\mbox{a.s.}
\end{equation}
for all $\beta>1$. Then, for any $\lambda'>\lambda$,
\begin{equation} \label{eqn:upper1ml}
\liminf_{t\to0}\frac{\Vert X\Vert_{t}}{b_{\lambda'}(t)}\leq
1\qquad\mbox{a.s.}
\end{equation}

\end{lemma}

\begin{pf}
For $\lambda'>\lambda$, we choose $\beta>1$ such that $\lambda
'>\lambda\beta$. First note that (\ref{mladen}) implies
\begin{equation} \label{eqn:suparg1}
\limsup_{n\to\infty}\frac{\Vert X\Vert_{(n+1)^{-(n+1)^{\beta
}}}}{b_{\lambda'}
(n^{-n^{\beta}} )}=0\qquad\mbox{a.s.},
\end{equation}
as $b_{\lambda}(t)$ is an increasing function in $\lambda$ for fixed
$t\geq0$.
Using the L\'evy property, we see the following:
\begin{eqnarray*}
&& \sum_{n}\P\Bigl(\sup_{(n+1)^{-(n+1)^{\beta}}\leq t <
n^{-n^{\beta}}}\bigl|X_{t}-X_{(n+1)^{-(n+1)^{\beta}}}\bigr|\leq
b_{\lambda'}
(n^{-n^{\beta}} ) \Bigr)\\
&&\quad=\sum_{n}\P\bigl(\Vert X\Vert_{n^{-n^{\beta
}}-(n+1)^{-(n+1)^{\beta}}}\leq
b_{\lambda'} (n^{-n^{\beta}} ) \bigr)\\
&&\quad\geq\sum_{n}\P\bigl(\Vert X\Vert_{n^{-n^{\beta}}}\leq
b_{\lambda'}
(n^{-n^{\beta}} ) \bigr)=\infty.
\end{eqnarray*}
The last step follows as in (\ref{est}) since now $\lambda\beta
/\lambda
'<1$. The Borel--Cantelli lemma shows that the sequence of independent events
\[
A_{n}= \Bigl\{\sup_{(n+1)^{-(n+1)^{\beta}}\leq t < n^{-n^{\beta
}}}\bigl|X_{t}-X_{(n+1)^{-(n+1)^{\beta}}}\bigr|\leq b_{\lambda'}
(n^{-n^{\beta}}
) \Bigr\}
\]
satisfies $\P(A_{n} \mbox{ i.o.})=1$. To reduce to the supremum, note
that
\[
\frac{\Vert X\Vert_{n^{-n^{\beta}}}}{b_{\lambda'}(n^{-n^{\beta
}})}\leq\frac
{\sup_{(n+1)^{-(n+1)^{\beta}}\leq t < n^{-n^{\beta
}}}|X_{t}-X_{(n+1)^{-(n+1)^{\beta}}}|}{b_{\lambda'}(n^{-n^{\beta
}})}+\frac{2\Vert X\Vert_{(n+1)^{-(n+1)^{\beta}}}}{b_{\lambda
'}(n^{-n^{\beta}})},
\]
and therefore, by (\ref{eqn:suparg1}),
\[
\liminf_{n\to\infty}\frac{\Vert X\Vert_{n^{-n^{\beta
}}}}{b_{\lambda
'}(n^{-n^{\beta}})}\leq\liminf_{n\to\infty}\frac{\sup
_{(n+1)^{-(n+1)^{\beta}}\leq t < n^{-n^{\beta
}}}|X_{t}-X_{n^{-(n+1)^{\beta}}}|}{b_{\lambda'}(n^{-n^{\beta}})}\leq1.
\]
This shows (\ref{eqn:upper1ml}).
\end{pf}

Now we are in position to prove Theorem~\ref{t3}. For a detailed
analysis of the $\limsup$ case, we refer to Savov \cite{S09}.

\begin{pf*}{Proof of Theorem~\ref{t3}}
The claim follows from Lemmas~\ref{lem:lower} and~\ref{L2}. To verify
the use of Lemma~\ref{L2} we still need to check that condition (\ref
{mladen}) holds for all $\beta>1$.

We fix $\beta>1$ and $\lambda'_2>\lambda_2$. Since $\lambda'_2$ is
fixed, we set $b:=b_{\lambda'_2}$ in order to increase readability. We
define the auxiliary function
\[
h(t)=b(\phi(t)),
\]
where $\phi(t)$ is chosen such that $\phi((\frac
{t}{t+1})^{((t+1)/t)^{\beta}} )=t^{1/t^{\beta}}$ and $\phi(0)=0$. Note
that $\phi$ is increasing and that $\phi(s^{-s^\beta
})=(s-1)^{-(s-1)^\beta}$. We also do not record that $\phi$ and $h$
depend on $\beta$ and $\lambda_2'$.

\textit{Step 1}: We show that
\begin{equation}\label{int}
\int_0^{1/2} \bar\Pi( h(t))\,\dd t<\infty.
\end{equation}

First, by the definition of $h$ and a change of variables, we obtain
\begin{eqnarray*}
&&\int_0^{1/2} \bar\Pi( h(t))\,\dd t\\
&&\quad=\int_0^{ C(\beta)} \bar\Pi( b (s^{s^{-\beta}} ) )\frac
{\dd(s/(s+1))^{((s+1)/s)^{\beta}}}{\dd s}\\
&&\quad=\int_0^{ C(\beta)} \bar\Pi( b (s^{s^{-\beta}} ) ) \biggl
(\frac{s}{s+1}
\biggr)^{((s+1)/s)^{\beta}}\biggl(\frac{s+1}{s}\biggr)^{\beta
-1}s^{-2}\bigl(1-\beta\log\bigl(1-(s+1)^{-1}\bigr)\bigr)\,\dd s,
\end{eqnarray*}
which can be estimated from above by
\begin{eqnarray*}
&&C\int_0^{ C(\beta)} \frac{b^{2}(s^{s^{-\beta}})\bar\Pi(
b(s^{s^{-\beta}}))}{b^{2}(s^{s^{-\beta}})} \biggl(\frac{s}{s+1}
\biggr)^{((s+1)/s)^{\beta}}s^{-1-\beta} |\log s|\,\dd s\\
&&\quad\leq C\int_0^{ C(\beta)} \frac{U( b(s^{s^{-\beta
}}))}{b^{2}(s^{s^{-\beta}})} \biggl(\frac{s}{s+1} \biggr
)^{((s+1)/s)^{\beta}}s^{-1-\beta} |\log s|\,\dd s\\
&&\quad\leq C'\int_0^{ C(\beta)} F( b(s^{s^{-\beta}})) \biggl
(\frac{s}{s+1}
\biggr)^{((s+1)/s)^{\beta}}s^{-1-\beta}|\log s|\,\dd s\\
&&\quad=\frac{C'}{\lambda}\int_0^{ C(\beta)}\frac{\log|\log
s^{s^{-\beta}}
|}{s^{s^{-\beta}}} \biggl(\frac{s}{s+1} \biggr)^{((s+1)/s)^{\beta
}}s^{-1-\beta}|\log s|\,\dd s\\
&&\quad\leq\frac{C'}{\lambda}\int_0^{ C(\beta)}s^{-1-\beta} (
\log |\log
s^{s^{-\beta}} | ) s^{((s+1)/s)^{\beta}-1/s^{\beta}}
|\log s|\,\dd s<\infty,
\end{eqnarray*}
where we have used $x^{2}\bar\Pi(x)\leq x^{2}\bar\Pi(x)+\int
_{-x}^{x}y^{2}\Pi(\dd y) + \sigma^2=U(x)\leq c x^2F(x)$ for some
absolute $c>0$ by Lemma~\ref{lem:flargeru} and the definition of $b$.

\textit{Step 2}: We denote by
\begin{equation} \label{eqn:defnAn}
A_{n}:= \bigl\{\mbox{there is at least one jump with modulus
$>b(n^{-n^{\beta}})$ up to time $(n+1)^{-(n+1)^{\beta}}$} \bigr\}
\end{equation}
and show that
\begin{equation}\label{ba}
\sum_n \P(A_{n} )<\infty.
\end{equation}
This comes from (\ref{int}). Indeed, note that $h$ inherits the
monotonicity of $b$ and $\phi$, and hence (\ref{int}) implies that
\begin{eqnarray} \label{eqn:nozerosu}
&&\sum_{n}\bigl((n+1)^{-(n+1)^{\beta}}-(n+2)^{-(n+2)^{\beta}} \bigr
) \bar\Pi\bigl(
h \bigl((n+1)^{-(n+1)^{\beta}} \bigr) \bigr)\nonumber\\ [-8pt]\\ [-8pt]
&&\quad\leq\sum_{n}\int
_{(n+2)^{-(n+2)^{\beta}}}^{(n+1)^{-(n+1)^{\beta}}}\bar\Pi(h(t))\,
\dd
t<\infty.\nonumber
\end{eqnarray}
Using
\begin{eqnarray*}
(n+1)^{-(n+1)^{\beta}}-(n+2)^{-(n+2)^{\beta}}&\sim
&(n+1)^{-(n+1)^{\beta
}},\\
b(n^{-n^{\beta}})&=&h\bigl((n+1)^{-(n+1)^{\beta}}\bigr),
\end{eqnarray*}
and that the sequence $(n+1)^{-(n+1)^{\beta}}\bar\Pi(
h((n+1)^{-(n+1)^{\beta}}))$ tends to zero by (\ref{eqn:nozerosu}), we
obtain that
\[
\P(A_{n} )=1-\mathrm{e}^{-(n+1)^{-(n+1)^{\beta}}\bar\Pi(
b(n^{-n^{\beta
}}))}\sim(n+1)^{-(n+1)^{\beta}}\bar\Pi\bigl(h \bigl
((n+1)^{-(n+1)^{\beta}} \bigr) \bigr)
\]
is summable. Therefore (\ref{ba}) is proved.

\textit{Step 3}: Let us now show how to use (\ref{ba}) to deduce
(\ref{mladen}). Obviously, it suffices to show that
\[
\limsup_{n\rightarrow\infty}\frac{\Vert X\Vert
_{(n+1)^{-(n+1)^{\beta
}}}}{b(n^{-n^{\beta}})}<\eps\qquad \mbox{a.s.}
\]
for any $\eps>0$ and, hence, by the Borel--Cantelli lemma, it suffices
to show that
\[
\sum_n \P\bigl({\Vert X\Vert_{(n+1)^{-(n+1)^{\beta}}}}>\eps{b
(n^{-n^{\beta}} )}
\bigr)<\infty.
\]
Separating jumps of absolute value larger or smaller than $ b
(n^{-n^{\beta}} )$, and, using the definition of $A_n$ in (\ref
{eqn:defnAn}), we obtain that
\begin{eqnarray*}
&&\sum_n \P\bigl({\Vert X\Vert_{(n+1)^{-(n+1)^{\beta}}}}>\eps
{b(n^{-n^{\beta}})} \bigr)\\
&&\quad=\sum_n \P\bigl({\Vert X\Vert_{(n+1)^{-(n+1)^{\beta
}}}}>\eps{b(n^{-n^{\beta}})}
; A^{c}_{n} \bigr) + \sum_n \P\bigl({\Vert X\Vert
_{(n+1)^{-(n+1)^{\beta
}}}}>\eps{b(n^{-n^{\beta}})} ; A_{n} \bigr),
\end{eqnarray*}
which is bounded from above by
\[
\sum_n \P\bigl({\Vert X\Vert_{(n+1)^{-(n+1)^{\beta}}}}>\eps
{b(n^{-n^{\beta}})}
|A^{c}_{n} \bigr)\cdot\P(A^{c}_{n} )+\sum_n
\P(A_{n} ).
\]
The second term is finite by (\ref{ba}); and the first term is bounded by
\begin{equation} \label{eqn:gaptoolarge}
\sum_n \P\bigl(\Vert X\Vert_{(n+1)^{-(n+1)^{\beta}}}>\eps
b(n^{-n^{\beta}})
|A^{c}_{n} \bigr).
\end{equation}
To estimate this sum note that conditionally on $A^{c}_{n}$,
$X_t\stackrel{d}=X_{t}(n)$, where $X(n)$ differs from $X$ only by
removing jumps of size larger than $|b(n^{-n^{\beta}})|$. Clearly, by
Wald's identity,
\begin{equation} \label{g}
\var(X_{t}(n))=t \biggl(\int_{-b(n^{-n^{\beta}})}^{b(n^{-n^{\beta
}})}y^{2}\Pi(\dd y)+\sigma^{2} \biggr) \leq t U(b(n^{-n^{\beta}})).
\end{equation}
Note that
\[
\bigl|\E X_{(n+1)^{-(n+1)^{\beta}}}(n)\bigr| = (n+1)^{-(n+1)^{\beta
}} \biggl| \int
_{|x|>b(n^{-n^\beta})} x \Pi(\dd x) - \gamma\biggr|.
\]
Therefore, by assumption (\ref{eqn:conditionM}), taking also into
account that $|\E X_{t}(n)|=t|\E X_{1}(n)|$, we obtain
\[
\sup_{t\leq(n+1)^{-(n+1)^{\beta}}}|\E X_t(n)|=\bigl|\E
X_{(n+1)^{-(n+1)^{\beta}}}(n)\bigr| = \mathrm{o}( b(n^{-n^{\beta}}) ).
\]
Using the previous relation (first step), Doob's martingale inequality
(second step), (\ref{g}) (third step), Lemma~\ref{lem:flargeru} (fourth
step) and the definition of $b$ (fifth step), we are led to the upper
bound of the term in (\ref{eqn:gaptoolarge}),
\begin{eqnarray*}
&&\sum_n \P\bigl({\Vert X(n)\Vert_{(n+1)^{-(n+1)^{\beta}}}}> \eps
{b(n^{-n^{\beta}})} \bigr)\\
&&\quad\leq\sum_n \P\biggl({\Vert X(n) - \E X(n)\Vert
_{(n+1)^{-(n+1)^{\beta}}}}>
\frac{1}{2} \eps{b(n^{-n^{\beta}})} \biggr)\\
&&\quad\leq\sum_n \frac{ 4 \E|X_{(n+1)^{-(n+1)^{\beta}}}(n) - \E
X_{(n+1)^{-(n+1)^{\beta}}}(n)|^2 }{(\eps/2)^2 b(n^{-n^{\beta}})^2}\\
&&\quad\leq\sum_n\frac{4 {(n+1)^{-(n+1)^{\beta}}}U (b(n^{-n^{\beta
}}) )}{(\eps/2)^2b(n^{-n^{\beta}})^2}\\
&&\quad\leq\sum_n\frac{4 {(n+1)^{-(n+1)^{\beta}}}C\cdot F
(b(n^{-n^{\beta
}}) )}{(\eps/2)^2}\\
&&\quad=\frac{C'}{\lambda\eps^2}\sum_n\frac{{(n+1)^{-(n+1)^{\beta
}}}\log
|\log n^{-n^{\beta}}|}{n^{-n^{\beta}}}<\infty,
\end{eqnarray*}
where we used the definition of $b$ in the last step. Thus, the term in
(\ref{eqn:gaptoolarge}) is finite, as required.
\hspace*{1pt}
\end{pf*}

\begin{pf*}{Proof of Corollary~\ref{cor:t1corollary}}
If $F$ is regularly varying so is $b_{\lambda}$; see Bingham, Goldie
and Teugels \cite{bgt}, Proposition~1.5.7. Now note that if $F$ is
regularly varying with exponent $-\alpha<0$, we have
\begin{eqnarray*}
b_{\lambda}(t)&=&F^{-1}(\log|\log t|/\lambda t)\\
&\sim&\lambda^{{1/\alpha}} F^{-1}(\log|\log t|/ t)\\
&=&\lambda^{{1/\alpha}}b_1(t).
\end{eqnarray*}
Hence, the statement of Theorem~\ref{t3} reads
\[
(\lambda'_1)^{{1/\alpha}}\leq\liminf_{t\rightarrow0}\frac
{\Vert X\Vert_{t}}{b_{1}(t)}\leq(\lambda'_2)^{{1/\alpha}}\qquad
\mbox{a.s.}
\]
for all $\lambda_1'<\lambda_1$ and $\lambda_2'>\lambda_2$. Taking the
limits on both sides, we obtain
\[
(\lambda_1)^{{1/\alpha}}\leq\liminf_{t\rightarrow0}\frac
{\Vert X\Vert_{t}}{b_{1}(t)}\leq(\lambda_2)^{{1/\alpha}}\qquad
\mbox{ a.s.}
\]
Applying the regular variation argument in the reverse direction yields
the claim.
\end{pf*}

\begin{pf*}{Proof of Corollary~\ref{cor:sddirectsymmetric}}
This follows directly from Theorem~\ref{t3}. The bounds on the
constants can be obtained from the absolute constants in
Proposition~\ref{prop:ad}.
\end{pf*}

\begin{pf*}{Proof of Theorem~\ref{t}}
Lemma~\ref{lem:lower} gives the lower LIL of the theorem.
Unfortunately, the arguments for the proof of Theorem~\ref{t3} do not
apply here. Hence, for the reverse direction, we show more directly
that the given norming function of the LIL implies the rate function of
the small deviations. The following arguments go back to Kesten. The
proof is via contradiction, assuming that
\begin{equation}\label{eqn:ass}
\liminf_{t\to0}\frac{\Vert X\Vert_{t}}{b_{\lambda_2'}(t)}>\frac
{2}{C}+\delta
\end{equation}
for some $\delta>0$ and $\lambda_2'>\lambda_2$. We show that under this
assumption we can derive, for sufficiently large $l$, the estimates
\begin{eqnarray}\label{eqn:schonwiedera}
1&\geq&\sum_{n\geq l}\P\biggl(\frac{\Vert X\Vert
_{r^{j}-r^{n}}}{b_{\lambda
_2'}(r^{j}-r^{n})}>\frac{2}{C};\mbox{for all $l\leq j\leq n-1$}
\biggr)\P
\bigl(\Vert X\Vert_{r^{n}}\leq b_{\lambda_2'}(r^{n}) \bigr)\\\label
{eqn:schonwiederb}
&\geq&\frac{1}{2}\sum_{n\geq l}\P\bigl(\Vert X\Vert_{r^{n}}\leq
b_{\lambda
_2'}(r^{n}) \bigr)
\end{eqnarray}
which is a contradiction as, by the choice of $b_{\lambda_2'}$ and the
small deviation\vspace*{1pt} rate (\ref{eqn:yetanothersdestimate}), the sum in
(\ref
{eqn:schonwiederb}) is infinite. First, let us derive estimate (\ref
{eqn:schonwiedera}) for which Assumption (\ref{eqn:ass}) is not needed.
For any fixed integer $l$ partitioning the probability space, we obtain
\begin{eqnarray*}
1&\geq&\sum_{n\geq l}\P\bigl(\Vert X\Vert_{r^{j}}>b_{\lambda
_2'}(r^{j})\mbox{ for
all $l\leq j\leq n-1$};\Vert X\Vert_{r^{n}}\leq b_{\lambda_2'}(r^{n})
\bigr)\\
&\geq&\sum_{n\geq l}\P\Bigl(\sup_{r^{n}\leq
s<r^{j}}|X_{s}|>b_{\lambda
_2'}(r^{j})\mbox{ for all $l\leq j\leq n-1$};\Vert X\Vert_{r^{n}}\leq
b_{\lambda_2'}(r^{n}) \Bigr).
\end{eqnarray*}
In order to employ the independence of increments of $X$ we estimate
from below by
\[
\sum_{n\geq l}\P\Bigl(\sup_{r^{n}\leq s<r^{j}}|X_{s}-X_{r^{n}}|>2
b_{\lambda_2'}(r^{j})\mbox{ for all $l\leq j\leq
n-1$};\Vert X\Vert_{r^{n}}\leq b_{\lambda_2'}(r^{n}) \Bigr)
\]
which equals
\begin{eqnarray*}
&&\sum_{n\geq l}\P\bigl(\Vert X\Vert_{r^{j}-r^{n}}>2 b_{\lambda
_2'}(r^{j})\mbox{ for all $l\leq j\leq n-1$} \bigr)\P
\bigl(\Vert X\Vert_{r^{n}}\leq b_{\lambda_2'}(r^{n}) \bigr)\\
&&\quad=\sum_{n\geq l}\P\biggl(\frac{\Vert X\Vert
_{r^{j}-r^{n}}}{b_{\lambda
_2'}(r^{j}-r^{n})}>2 \frac{b_{\lambda_2'}(r^{j})}{b_{\lambda
_2'}(r^{j}-r^{n})}\mbox{ for all $l\leq j\leq n-1$} \biggr)\P
\bigl(\Vert X\Vert_{r^{n}}\leq b_{\lambda_2'}(r^{n}) \bigr).
\end{eqnarray*}
By the monotonicity of $b_{\lambda_2'}$, this yields the lower bound
\[
\sum_{n\geq l}\P\biggl(\frac{\Vert X\Vert
_{r^{j}-r^{n}}}{b_{\lambda
_2'}(r^{j}-r^{n})}>2 \frac{b_{\lambda_2'}(r^{j})}{b_{\lambda
_2'}(r^{j}-r^{j+1})};\mbox{for all $l\leq j\leq n-1$} \biggr)\P
\bigl(\Vert X\Vert_{r^{n}}\leq b_{\lambda_2'}(r^{n}) \bigr).
\]
Finally, we utilize the regularity of $b_{\lambda_2'}$ from (\ref
{eqn:regularityofb}) to obtain the lower bound
\[
\sum_{n\geq l}\P\biggl(\frac{\Vert X\Vert
_{r^{j}-r^{n}}}{b_{\lambda
_2'}(r^{j}-r^{n})}>\frac{2}{C};\mbox{for all $l\leq j\leq n-1$}
\biggr)\P
\bigl(\Vert X\Vert_{r^{n}}\leq b_{\lambda_2'}(r^{n}) \bigr).
\]
As required, we derived Estimate (\ref{eqn:schonwiedera}).

Assuming (\ref{eqn:ass}) we now derive Estimate (\ref
{eqn:schonwiederb}). The assumption directly shows that
\[
\lim_{t\to0}\P\biggl(\bigcap_{s\leq t} \{\Vert X\Vert_{s}\geq2C^{-1}
b_{\lambda_2'}(s) \} \biggr)=1
\]
which implies that we may choose $l$ large enough such that
\[
\P\biggl(\frac{\Vert X\Vert_{r^{j}-r^{n}}}{b_{\lambda
_2'}(r^{j}-r^{n})}>\frac
{2}{C};\mbox{for all $l\leq j\leq n-1$} \biggr)\geq\P\biggl
(\bigcap_{s\leq r^{l}}
\{\Vert X\Vert_{s}\geq2C^{-1} b_{\lambda_2'}(s) \} \biggr)\geq
\frac{1}{2}.
\]
Hence, we derived estimate (\ref{eqn:schonwiederb}) so that the proof
is complete.
\end{pf*}

\begin{pf*}{Proof of Corollary~\ref{cor:t1corollary2}}
This is
completely analogous to the proof of Corollary~\ref{cor:t1corollary}.
\end{pf*}

\begin{pf*}{Proof of Corollary~\ref{cort2}}
We use Proposition~\ref
{prop:ad} and Theorem~\ref{t3}. In order to do so, we have to see that
the term $\eps u_\eps$ in (\ref{eqn:sdquantityadstrengthend}) has no
influence on the order. We apply Lemma~\ref{lem:lower} and follow the
proof of Theorem~\ref{t} with the scaling
\[
t=r^n \quad\mbox{and}\quad\eps=b(r^n)
\]
and with the sequence $n^{-n^\beta}$, respectively. Therefore, it is
sufficient to show that
\[
\eps u_\eps=\mathrm{o}( t F(\eps))
\]
with the above scalings of $t$ and $\eps$. Since $\eps=b(t)$ and thus
$t\sim F(\eps)^{-1} \log\log F(\eps)$, we need to show that
\[
\eps u_\eps=\mathrm{o}( \log\log F(\eps)).
\]
As this is precisely what we stated in condition (\ref
{eqn:cond-esschervanishes}), the proof is complete.
\end{pf*}

\begin{pf*}{Proof of Proposition~\ref{prop:bvnodrift}}
As $X$ is of bounded variation, the representation
\[
X_t=A^1_t-A^2_t+ct
\]
holds with two independent pure jump subordinators $A^1, A^2$. Next, we
use the simple observation
\[
\frac{|X_t|}{t}\leq\frac{\Vert X\Vert_t}{t}\leq\frac
{\Vert A^1\Vert_t+\Vert A^2\Vert_t+|c|t}{t}=\frac{A^1_t}{t}+\frac
{A^2_t}{t}+|c|
\]
to conclude the proof. The left-hand side converges to $|c|$, as $X$
has bounded variation (see Theorem~39 of Doney \cite{D}). Finally, the
right-hand side converges to $|c|$ as $|A^i_t|/t$ converge at zero
almost surely to their drift (see Proposition~5 of Doney \cite{D}).
\end{pf*}


\section*{Acknowledgements}
The first author was supported by the DFG Emmy Noether Programme. The
second author was supported by the EPSRC Grant EP/E010989/1. The third
author acknowledges the support of New College, Oxford. We thank Thomas
Simon (Lille) for pointing out the reference Wee \cite{W88} to us.

%

\printhistory

\end{document}